\documentclass[12pt]{article}
\usepackage[left=1in,top=1in,right=1in,bottom=1in,nohead]{geometry}

\usepackage{url,subfig,latexsym,amsmath,amssymb, amsfonts,enumerate,
  epsfig, graphics, amsthm,mathtools}
\usepackage{enumitem}
\usepackage{color}
\usepackage{nameref}

\usepackage[colorlinks]{hyperref}

\newtheorem{thm}{Theorem}[section]
\newtheorem{cor}[thm]{Corollary}
\newtheorem{lem}[thm]{Lemma}

\newtheorem*{thm*}{Theorem}
\newtheorem{innercustomthm}{Theorem}
\newenvironment{customthm}[1]
  {\renewcommand\theinnercustomthm{#1}\innercustomthm}
  {\endinnercustomthm}

\newtheorem{defn}{Definition}[section]

\newtheorem{example}{Example}[section]

\newcommand{\thmref}[1]{Theorem~\ref{#1}}

\newcommand{\lemref}[1]{Lemma~\ref{#1}}

\newcommand{\dlim}{\displaystyle \lim\limits}
\newcommand{\dsum}{\displaystyle \sum\limits}

\newcommand{\dprod}{\displaystyle \prod\limits}

\def\S{\mathcal S}
\def\C{\mathcal C}
\def\Re{\mathcal R}
\def\R{\mathbb R}
\def\L{\mathcal L}
\def\Z{\mathbb Z}
\def\A{\mathcal A}
\def\Td1{T^{D,1}_{\{x_n\}}}
\def\Ts1{T^{S,1}_{\{x_n\}}}
\def\Des{D_{\{x_n\}}}
\def\lmt{\lim_{n\rightarrow \infty}}
\def\Tsinf{T^{S,\infty}_{\{x_n\}}}

\def\Tdi{T^{D,i}_{\{x_n\}}}

\newcount\colveccount
\newcommand*\colvec[1]{
        \global\colveccount#1
        \begin{pmatrix}
        \colvecnext
}
\def\colvecnext#1{
        #1
        \global\advance\colveccount-1
        \ifnum\colveccount>0
                \\
                \expandafter\colvecnext
        \else
                \end{pmatrix}
        \fi
}

\title{Some network conditions for positive recurrence  of stochastically modeled reaction networks}

\author{
David F. Anderson\thanks{Department of Mathematics, University of
  Wisconsin, Madison, USA.  anderson@math.wisc.edu, grant support from NSF-DMS-1318832 and Army Research Office grant W911NF-14-1-0401.},
\and
Jinsu Kim\thanks{Department of Mathematics, University of
  Wisconsin, Madison, USA.  jskim@math.wisc.edu, grant support from NSF-DMS-1318832 and Army Research Office grant W911NF-14-1-0401.}
  }

\begin{document}

\maketitle

\begin{abstract}
We consider discrete-space continuous-time Markov models of reaction networks and provide sufficient conditions for the following stability condition to hold: each state in a closed, irreducible component of the state space is positive recurrent; moreover the time required  for a trajectory to enter such a component has finite expectation. 
The provided analytical results depend solely on the underlying structure of the reaction network and not on the specific choice of model parameters.   Our main results apply to binary systems and our main analytical tool is the ``tier structure'' previously utilized successfully in the study of deterministic models of reaction networks.  

\end{abstract}

\section{Introduction}

Mathematical models of reaction networks generally fall into two categories.  When the counts of the constituent species are high a (usually nonlinear) set of ordinary differential equations is  used to model their concentrations.  However, when the abundances are low a continuous-time Markov chain is often used to model the counts of the different species.  In this paper we consider the stochastic model and provide conditions on the associated reaction network  that ensure the Markov chain satisfies the following: each state in a closed, irreducible component of the state space is positive recurrent; moreover, if  $\tau_{x_0}$ is the time for the process to enter the union of the closed irreducible components given an initial condition  $x_0$, then $\mathbb{E}[\tau_{x_0}] < \infty$.  Note, therefore, that even if the initial condition of the process is a transient state, the trajectory will eventually enter a positive recurrent component of the state space.   
 Importantly, the main analytical results provided here hold regardless of  the choice of rate parameters for the model, which are often unknown.

This work falls into the broad research area known as chemical reaction network theory, which dates back to at least \cite{Horn72,HornJack72} where graphical characteristics of networks were shown to ensure uniqueness and local asymptotic stability of the steady states for deterministically modeled complex-balanced systems.   Since that time, much of the focus of chemical reaction network theory has been related to discovering how the qualitative properties of deterministic models relate to their reaction networks \cite{AndGAC_ONE2011,AndBounded_ONE2011,Sontag2007,AMC-Badal2013,lotkareactions,
ChavesEduardo2002,Craciun2005,CraciunPantea,Feinberg72, Feinberg87,FeinbergLec79, feedbackvertex1,GMS13a, J-S1,signforinjectivity,Pantea2012}.  However, with the advent of new technologies   -- most notably fluorescent proteins -- there is now a large literature demonstrating that the fluctuations arising from the effective randomness of individual interactions in cellular systems can have significant consequences on the emergent behavior  of the system \cite{Arkin1998, Becskei2005, Elowitz2002, Paulsson2011,Maamar2007, Paulsson2004, ULPOSP2016}. Hence, analytical results related to  stochastic models are essential if these systems are to be well understood, and attention is shifting in their direction.

In the deterministic modeling regime there are a number of network conditions that guarantee some sort of stability for the model.  These conditions include weak reversibility and deficiency zero \cite{Feinberg72, Feinberg87,FeinbergLec79}, weak reversibility and a single linkage class \cite{AndGAC_ONE2011,AndBounded_ONE2011}, endotactic \cite{CraciunPantea}, strongly endotactic \cite{GMS13a}, tropically endotactic \cite{Jim_Gheorghe2017}, etc.  However, to the best of the authors' knowledge, in the stochastic context there is only one such result: in \cite{AndProdForm} a model whose reaction network is weakly reversible and has a deficiency of zero is shown to be positive recurrent (and the stationary distribution is characterized as a product of Poissons).   (The paper \cite{scalable2014} provides sufficient conditions for positive recurrence, but the provided conditions are analytical in nature and do not explicitly relate to the network structure of the model.)
Here we provide network conditions guaranteeing  stability for two  classes of binary models: (i) models that are weakly reversible, have a single linkage class, and have in-flows and out-flows, and (ii)  a new category of networks we term ``double-full.''  We will show that all states are positive recurrent for the first class of models.  We will show that for the second class of models the stability condition detailed in the opening paragraph holds: all states in closed, irreducible components of the state space are positive recurrent, and every trajectory enters such a component in finite time.    Our main analytic tools are Lyapunov functions and ideas related to ``tier structures'' as introduced in \cite{AndGAC_ONE2011,AndBounded_ONE2011}, and also utilized in \cite{GMS13a}.

 The outline of this paper is as follows. In Section \ref{sec:mathmodel}, we introduce the relevant mathematical model, including the required terminology from chemical reaction network theory, and provide a statement of the main results.   In Section \ref{sec:tier}, we generalize the ``tier structure'' developed in \cite{AndGAC_ONE2011,AndBounded_ONE2011} and present some preliminary analytical results.   In Section \ref{sec:main}, we provide a general result relating the tiers of a reaction network to the stability of the Markov model. In Section \ref{sec:applications}, we restate  our main results and provide their proofs.  We also introduce some generalizations of our main results. 
  Finally, in Section \ref{sec:discussion}, we provide a brief discussion and directions for future research.  In particular, we discuss how we  believe weak reversibility alone guarantees positive recurrence, and that this conjecture, which is similar to the Global Attractor Conjecture for deterministically modeled systems,  will be an active area of future research.

\section{Mathematical model and a statement of the main results}
\label{sec:mathmodel}

In Section \ref{sec:reactionnetworks}, we formally introduce reaction networks.  In Section \ref{sec:dynamics}, we introduce both the  deterministic  and stochastic models.
In Section \ref{sec:mainresults}, we state our main results.

\subsection{Reaction networks}
\label{sec:reactionnetworks}

A reaction network is a graphical construct that describes the set of possible interactions among the constituent species.  

\begin{defn}\label{def:21}
\textbf{A  reaction network} is given by a triple of finite sets $(\S,\C,\Re)$ where:
\begin{enumerate}
\item \textbf{The species set} $\S=\{S_1,S_2,\cdots,S_d\}$ contains the species of the reaction network.
\item \textbf{The reaction set} $\Re=\{R_1,R_2,\cdots,R_r\}$ consists of ordered pairs $(y,y') \in \Re$ where 
\begin{align}\label{complex}
y=\sum_{i=1}^d y_iS_i \hspace{0.4cm} \textrm{and} \hspace{0.4cm}
y'=\sum_{i=1}^d y'_iS_i
\end{align}
and where the values $y_i,y'_i \in \mathbb{Z}_{\ge 0}$ are the \textbf{stoichiometric coefficients}. We will often write reactions $(y,y')$ as $y\rightarrow y'$.
\item \textbf{The complex set} $\C$ consists of the linear combinations of the species in (\ref{complex}). Specifically, 
$\C = \{y\ |\ y\rightarrow y' \in \Re\} \cup \{y' \ |\ y\rightarrow y' \in \Re\}$. For the reaction $y\to y'$, the complexes $y$ and $y'$ are termed the \textbf{source} and \textbf{product} complex of the reaction, respectively.
\end{enumerate}
\end{defn}
Allowing for a slight abuse of notation, we will let $y$ denote both the linear combination in \eqref{complex} and the vector whose $i$th component is $y_i$,  i.e.~$y=(y_1,y_2,\cdots,y_d)^T \in \mathbb{Z}^d_{\ge 0}$. 
For example, when $\S=\{S_1,S_2,\dots,S_d\}$, $2S_1+S_2$ is associated with $(2,1,0,0,\dots,0) \in \Z^d_{\ge 0}$.

Note that it is perfectly valid to have a linear combination of the form \eqref{complex} with $y_i = 0$ for each $i$.  In this case, we denote the complex by $\emptyset$.

It is most common to present a reaction network with a directed \textit{reaction graph} in which the nodes are the complexes and the directed edges are given by the reactions.
We present an example to solidify notation.

\begin{example}
Consider the reaction network with associated reaction graph
\[
	S+E \rightleftharpoons SE \rightarrow E+P,
\]
which is a usual model   for substrate-enzyme kinetics. For this reaction network, $\S = \{S,E,SE,P\}$, $\C=\{S+E,SE,E+P\}$ and $\Re=\{S+E\rightarrow SE, SE\rightarrow S+E,SE\rightarrow E+P\}$. \hfill $\triangle$
\end{example}

\begin{defn}
Let $(\S,\C,\Re)$ be a reaction network.  The connected components of the associated reaction graph are termed \textbf{linkage classes}.   If a linkage class is strongly connected, then it is called \textbf{weakly reversible}. If all linkage classes in a reaction network are weakly reversible, then the reaction network is said to be weakly reversible. 
\end{defn}

\begin{example}\label{ex11}
Consider the reaction network with associated reaction graph
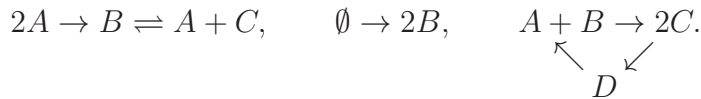
\begin{align*}
&2A \rightarrow B \rightleftharpoons A+C, \qquad \emptyset \rightarrow 2B, \qquad A+B \rightarrow 2C.\\[-1ex]
&  \hspace{7.1cm}  \displaystyle \nwarrow \hspace{.2in}  \swarrow \\[-1ex]
& \hspace{3.04in} D
\end{align*}
This network has three linkage classes.  The right-most linkage class is weakly reversible, whereas the other two are not.
\hfill $\triangle$
\end{example}

We will denote by $\S(\L),\C(\L)$, and $\Re(\L)$ the sets of  species, complexes, and reactions  involved in linkage class $\mathcal L$, respectively. 

The following definitions related to possible network structures are required to state our main results.

\begin{defn}
A reaction network $(\S,\C,\Re)$ is called \textbf{binary}  if $\sum_{i=1}^d y_i \le 2$ for all $y \in \C$.
\end{defn}

\begin{defn}
Let $(\S,\C,\Re)$ be a reaction network with $\S=\{ S_1,S_2,\cdots,S_d\}$.  
The complex $\emptyset$ is termed the \textbf{zero complex}.  Complexes of the form  $S_i$ are termed  \textbf{unary complexes} and complexes of the form   $S_i+S_j$  are termed \textbf{binary complexes}.  Binary complexes of the form $2S_i$ are termed \textbf{double complexes}.  If $2S_i\in \C$ for each $i=1,2,\dots,d$, then  the reaction network $(\S, \C, \Re)$ is \textbf{double-full}.
\end{defn}

\begin{defn}
We call the reactions $\emptyset \rightarrow S$ and $S \rightarrow \emptyset$  the \textbf{in-flow} and \textbf{out-flow} of $S$, respectively.  We say a reaction network \textbf{has all in-flows and out-flows} if $\emptyset \to S\in \Re$ and $S \to \emptyset \in \Re$ for each $S \in \S$.
\end{defn}

\subsection{Dynamical systems}
\label{sec:dynamics}

In this section, we introduce two dynamical models for reaction networks. We begin with the usual Markov chain model, and then present the deterministic model.

For the usual Markov model the vector  $X(t) \in \Z^d_{\ge 0}$ gives the counts of the constituent species at time $t$, and the transitions are determined by the reactions.  In particular, for appropriate state-dependent intensity (or rate)  functions $\lambda_{y\to y'}:\Z^d_{\ge 0} \to \R_{\ge 0}$ we assume that for each $y\to y'\in \Re$,
\begin{align}
P(X(t+\Delta t) = x+y'-y \ | \ X(t)=x) = \lambda_{y\rightarrow y'}(x)  \Delta t + o(\Delta t) \label{prob}.
\end{align} 
The generator $\mathcal{A}$ of the associated Markov process is \cite{Kurtz86}
\begin{align}\label{gen5}
 \mathcal{A}V(x) = \sum_{y \rightarrow y' \in \Re} \lambda_{y\rightarrow y'}(x)(V(x+y'-y)-V(x)), 
\end{align}
for a function $V:\Z^d_{\ge 0} \to \R$.
The usual choice of intensity is given by \textit{stochastic mass-action kinetics}
\begin{align}\label{mass}
\lambda_{y\rightarrow y'}(x)= \kappa_{y\rightarrow y'} \prod_{i=1}^d \frac{x_i !}{(x_i-y_{i})!}\mathbf{1}_{\{x_i \ge y_{i}\}},
\end{align} 
where the positive constant $\kappa_{y\rightarrow y'}$ is the  reaction rate constant.
We typically incorporate the rate constants into the reaction graphs by placing them next to the reaction arrow as in
 $y\xrightarrow{\kappa} y'$.  Trajectories of this model are typically simulated via the Gillespie algorithm \cite{Gill76,Gill77} or the next reaction method \cite{Anderson2007a,Gibson2000}, or are approximated via tau-leaping \cite{Anderson2007b,Gill2001}.

For  the deterministic model, we let the vector $x(t) \in \R^d_{\ge 0}$ solve 
\begin{align}
\dfrac{d}{dt}x(t) = \sum_{y\rightarrow y'}\kappa_{y\rightarrow y'} x(t)^y(y'-y), 
\end{align}
where for two vectors $u,v\in \R^d_{\ge 0}$, we define $u^v = \prod_{i=1}^d u_i^{v_i}$, with the convention $0^0 = 1$. The vector $x(t)$ then models  the concentrations of the constituent species at time $t$. See \cite{AndKurtz2011, AK2015,KurtzPop81} for the connection between the stochastic and deterministic models.  The choice of rate function, i.e.~$\kappa_{y\to y'} x(t)^y$, is termed \textit{deterministic mass-action kinetics.}

Recalling the discussion below Definition \ref{def:21}, we will sometimes write $x^{\sum_{i=1}^d y_iS_i}$ for $\prod_{i=1}^d x_i^{y_i}$.   For example, we have $x^{2S_1 + S_2} = x_1^2 x_2$.

\subsection{Statement of main results}
\label{sec:mainresults}

In this section we state our main results.  The first, Theorem \ref{thm:main1}, is similar to the main results found in \cite{AndGAC_ONE2011,AndBounded_ONE2011} related to deterministic systems that are weakly reversible and have a single linkage class.

\begin{thm}\label{thm:main1}
Let $(\S,\C, \Re)$ be a weakly reversible, binary reaction network that has a single linkage class.   Let $\widetilde \Re =  \Re \cup_{S\in \S} \{\emptyset \to S, S \to \emptyset\}$ and $\widetilde{\C} = \C \cup\{\emptyset\}\cup\{S \ | \ S \in \S\}$. Then, for any choice of rate constants, every state of the Markov process with intensity functions \eqref{mass} associated to the reaction network $(\S,\widetilde{\C},\widetilde \Re)$ is positive recurrent.

%
\end{thm}

Our second main result is related to models that are double-full.

\begin{thm}\label{thm52}
Let $(\S,\C, \Re)$ be a binary reaction network satisfying the following two conditions:
\begin{enumerate}
\item the reaction network is double-full, and
\item for each double complex (of the form $2S_i$) there is a directed path within the reaction graph beginning with the double complex itself and ending with either a unary complex (of the form $S_j$) or  the zero complex.
\end{enumerate} 
Then, for any choice of rate constants,  the Markov process with intensity functions \eqref{mass} associated to the reaction network $(\S,\C,\Re)$ satisfies the following: each state in a closed, irreducible component of the state space is positive recurrent; moreover, if  $\tau_{x_0}$ is the time for the process to enter the union of the closed irreducible components given an initial condition  $x_0$, then $\mathbb{E}[\tau_{x_0}] < \infty$.
\end{thm}

\section{Mathematical preliminaries: Lyapunov functions and tiers}\label{sec:tier}

In this section we introduce the well known Foster-Lyapunov conditions.  We also generalize the ``tier structure'' that was introduced in \cite{AndGAC_ONE2011,AndBounded_ONE2011} for continuous, deterministic models to the discrete, stochastic setting.  In Section \ref{sec:main}, we will use the new ideas related to tiers to ensure  the Foster-Lyapunov conditions hold for our models of interest.

 The theorem below relating Lyapunov functions with positive recurrence of a Markov model is well known.  See \cite{MT-LyaFosterIII} for more on this topic.


\begin{thm}\label{thm11}Let $X$ be a continuous-time Markov chain on a countable state space $\mathbb{S}$ with generator $\mathcal A$. Suppose there exists a finite set $K \subset \mathbb{S}$ and a positive function $V$ on $\mathbb{S}$ such that 
\begin{eqnarray}\label{lya}
\mathcal{A}V(x) \le -1
\end{eqnarray}
for all $x \in \mathbb{S}\setminus K$. Then each state in a closed, irreducible component of $\mathbb{S}$ is positive recurrent.  Moreover, if $\tau_{x_0}$ is the time for the process to enter the union of the closed irreducible components given an initial condition  $x_0$, then $\mathbb{E}[\tau_{x_0}] < \infty$.
\end{thm}

%
%
%

In this paper, our main Lyapunov function will be 
$V(x) = \sum_{i=1}^d v(x_i)$ where
\begin{align}\label{eq:MainLyapunov}
v(x)=
\begin{cases}
 x(\ln{x}-1)+1,  &\textrm{if} \ x \in \Z_{\ge 0}\\
 1, & \textrm{otherwise},
 \end{cases}
\end{align}  with the convention $0 \ln{0}=0$. This function has been used widely  to verify the stability of deterministic models of  reaction networks. In particular, it  played a significant role in the proof of the Deficiency Zero Theorem of chemical reaction network theory \cite{Feinberg72, Feinberg87,FeinbergLec79,HornJack72}.

We turn to the task of generalizing the tier structures introduced in \cite{AndGAC_ONE2011,AndBounded_ONE2011} to the discrete state space setting.  The key difference between the present setting and that of \cite{AndGAC_ONE2011,AndBounded_ONE2011} is that now the process can  hit the boundary of the state space, where intensity functions can take the value of zero.

%

We begin by  introducing some useful terminology.
\begin{enumerate}
\item  For $x\in \Z^d$, we write $(x \vee 1)$ for the vector in $\Z^d$ with $j$th component
$(x\vee 1)_j = x_j\vee 1 = \max\{x_j,1\}$. 
\item We will use the the phrase ``for large $n$" for ``for all $n$ greater than some fixed constant $N$".
\item For each complex $y \in \C$, we define the following  function, 
\begin{align*}
\lambda_y(x) := \prod_{i=1}^d \frac{x_i !}{(x_i-y_{i})!}\mathbf{1}_{\{x_i \ge y_{i}\}},\quad \text{ for } x \in \Z^d.
\end{align*}
Note that for the reaction $y\to y'\in \Re$,  we have $\lambda_{y\rightarrow y'}(x) = \kappa_{y\rightarrow y'}\lambda_y(x)$.  That is, $\lambda_y(x)$ is the portion of the stochastic mass-action term that  depends upon the source complex $y$.
\end{enumerate}
 
 We now define two types of tiers: D-type and S-type.  The D-type and S-type  tiers will allow us to later quantify the relative sizes of the different terms in the expression for $\mathcal AV$, where $V$ is defined in and around \eqref{eq:MainLyapunov}, required to utilized Theorem \ref{thm11}.  In particular, you can see in equation \eqref{asymptotic of AV} the precise functional role each term plays in $\mathcal A V$.   
\begin{defn}\label{def31}
Let $( \S,\C,\Re)$ be a  reaction network and let $\{x_n\}$ be a sequence in $\R^d_{\ge 0}$.  We say that $\C$ has a \textbf{D-type partition along  $\{x_n\}$} if there exist a finite number of nonempty mutually disjoint subsets $T^{D,i}_{\{x_n\}} \subset \C$ such that $\cup_{i} T^{D,i}_{\{x_n\}}  = \C$, and
\begin{enumerate}
\item if $y,y' \in T^{D,i}_{\{x_n\}}$, then there exists a $C \in (0,\infty)$ such that
\begin{align*}
\lim_{n\rightarrow \infty} \frac{(x_n \vee 1)^{y}}{(x_n\vee 1)^{y'}} = C, 
\end{align*}
\item if $y \in T^{D,i}_{\{x_n\}}$ and $y' \in T^{D,k}_{\{x_n\}}$ with $i < k$ then
\begin{align*}
\lim_{n\rightarrow \infty} \dfrac{(x_n \vee 1)^{y'}}{(x_n\vee 1)^{y}} = 0.
\end{align*}
\end{enumerate}
The mutually disjoint subsets $T^{D,i}_{\{x_n\}}$ are called \textbf{D-type tiers along $\{x_n\}$}.  We will say that $y$ is in a higher tier than $y'$ in the D-type partition along $\{x_n\}$ if $y \in T^{D,i}_{\{x_n\}}$ and $y' \in T^{D,j}_{\{x_n\}}$ with $i < j$. In this case, we will denote $y \succ_D y'$. If $y$ and $y'$ are in the same D-type tier, then we will denote this by $y\sim_D y'$. 
\end{defn} 

Note that $\C$ is a well-ordered set with $\succ_D$ and $\sim_D$.

The terminology of saying tier $\Td1$ is \textit{higher} than the other tiers comes from point 2 ~in Definition \ref{def31}. 

\begin{defn}\label{def32}
Let $( \S,\C,\Re)$ be a reaction network and let $\{x_n\}$ be a sequence in $\R^d_{\ge 0}$.   We say that $\C$ has a \textbf{S-type partition along  $\{x_n\}$} if there exist a finite number of nonempty mutually disjoint subsets $T^{S,i}_{\{x_n\}} \subset \C$,  with  $i\in \{1,\dots,P,\infty\}$, such that $T^{S,1}_{\{x_n\}}\cup \cdots \cup T^{S,P}_{\{x_n\}}\cup T^{S,\infty}_{\{x_n\}}  = \C$, and
\begin{enumerate}
\item $y \in T^{S,\infty}_{\{x_n\}}$ if and only if $\lambda_{y}(x_n) = 0$ for all $n$, 

\item $\lambda_y(x_n) \ne 0$ for any $n$ if $y \in T_{\{x_n\}}^{S,i}$ for $i \in \{1,\dots, P\},$

\item if $y,y' \in T^{S,i}_{\{x_n\}}$, with $i \in \{1,\dots,P\}$, then there exists a $C \in (0,\infty)$ such that
\begin{align*}
\lim_{n\rightarrow \infty} \dfrac{\lambda_{y}(x_n)}{\lambda_{y'}(x_n)} = C, 
\end{align*}
\item if $y \in T^{S,i}_{\{x_n\}}$ and $y' \in T^{S,k}_{\{x_n\}}$ with $1\le i < k\le P$, then
\begin{align*}
\lim_{n\rightarrow \infty} \dfrac{\lambda_{y'}(x_n)}{\lambda_{y}(x_n)} = 0.
\end{align*}
\end{enumerate}
The mutually disjoint subsets $T^{S,i}_{\{x_n\}}$ are called \textbf{S-type tiers along $\{x_n\}$}.  We will say that $y$ is in a higher tier than $y'$ in the S-type partition along $\{x_n\}$ if $y \in T^{S,i}_{\{x_n\}}$ and $y' \in T^{S,j}_{\{x_n\}}$ with $i < j$. 
\end{defn}

We present an example in order to clarify  the previous two definitions.

\begin{example} \label{ex: idea of tier}
Consider the reaction network
\[
A + B \rightleftharpoons C \rightleftharpoons \emptyset
\]
and let $x_n = (n^2,0,n)$.  For this sequence and this reaction network we have
\[
(x_n \vee 1)^{A+B} = n^2, \quad (x_n\vee 1)^C = n, \quad \text{and} \quad (x_n \vee 1)^\emptyset = 1,
\]
and so
\[
\Td1=\{A+B\}, \quad T^{D,2}_{\{x_n\}} = \{C\}, \quad T^{D,3}_{\{x_n\}} = \{\emptyset\}.
\]
However,
\[
\lambda_{A+B}(x_n) = 0, \quad \lambda_C(x_n) = n, \quad \text{and}\quad \lambda_\emptyset(x_n) = 1,
\]
so that
\[
T^{S,\infty}_{\{x_n\}}=\{A+B\}, \quad T^{S,1}_{\{x_n\}} = \{C\}, \quad T^{S,2}_{\{x_n\}} = \{\emptyset\}.
\]
\hfill $\triangle$
\end{example}

These two tier structures make hierarchies for the complexes  with respect to the sizes of $(x_n \vee 1)^y$ and $\lambda_y(x_n)$ along a sequence $\{x_n\}$.

\begin{defn}
Let $(\S,\C,\Re)$ be a reaction network and suppose that $\{T^{D,i}_{\{x_n\}} \}$ are D-type tiers along a sequence $\{x_n\}$.  We will call $y \rightarrow y'\in \Re$  a \textbf{descending reaction} along $\{x_n\}$ if $y \in T^{D,1}_{\{x_n\}}$ and $y'\in T^{D,i}_{\{x_n\}}$ for some $i > 1$. We denote
\begin{align*}
D_{\{x_n\}} := \{ y \in \C\ | \ y\rightarrow y' \ \textrm{is a descending reaction along $\{x_n\}$} \}.
\end{align*}
\end{defn}

For example, note that for the reaction network and sequence of Example \ref{ex: idea of tier}, the set of descending reactions along $\{x_n\}$ is $\{A + B \to C\}$ and so $D_{\{x_n\}} = \{A + B\}$.

\begin{defn}\label{def_tier-seqence}
For a network $(\S,\C,\Re)$ with $|\S|=d$, a sequence $\{x_n\} \subset \Z^d_{\ge 0}$ is called a \textbf{tier-sequence} if \begin{enumerate}
\item $\lim_{n\rightarrow\infty}x_{n,i} \in [0,\infty]$ for all $i=1,2,\cdots,d$ and  $\lim_{n\rightarrow \infty}x_{n,i} = \infty$ for at least one $i$, and
\item $\C$ has both a D-type partition and an S-type partition along $\{x_n\}$. 
\end{enumerate}
\end{defn}

Note that $\{x_n\}$ given in Example \ref{ex: idea of tier} is a tier-sequence for the given  reaction network.

\begin{lem}\label{lem21} Let $( \S,\C,\Re)$ be a reaction network with $|\S|=d$ and let $\{x_n\} \subset \mathbb{Z}^d_{\ge0}$ be an arbitrary sequence such that $\lim_{n\rightarrow \infty}|x_n|=\infty$. Then there exists a subsequence of $\{x_n\}$ which is tier-sequence.  
\end{lem}
The proof of \lemref{lem21} is similar to that of Lemma 4.2 in \cite{AndGAC_ONE2011}.  In particular, the relevant tiers can each be constructed sequentially by repeatedly taking  subsequences. The details are omitted for the sake of brevity.    

%

\begin{lem}\label{lem : Td1 goes infinity}
 Let $\{x_n\}$  be a tier-sequence of a reaction network $(\S,\C,\Re)$. If $y_0\in \Td1$, then $\lmt (x_n\vee 1)^{y_0}=\infty$.
\end{lem}
\begin{proof}
Let $I=\{i \ | \ \lmt x_{n,i}=\infty\}$ and $y\in \C$ such that $y_i\neq 0$ for some $i \in I$.  By  definition, if $y_0\in \Td1$ then
\begin{align*}
\lmt \frac{(x_n \vee 1)^{y}}{(x_n \vee 1)^{y_0}} = C
\end{align*}
for some constant $C\ge 0$. Since $\lmt (x_n\vee 1)^{y} =\infty$,
the result follows. 
\end{proof}

In the next lemma, we provide relations between D-type partitions and S-type partitions.

\begin{lem}\label{lem22}  Let $\{x_n\}$  be a tier-sequence of a reaction network $(\S,\C,\Re)$ and let $y\in \C$. Then
\begin{align*}
\lim_{n \rightarrow \infty} \dfrac{\lambda_{y}(x_n)}{(x_n \vee 1)^y} = \begin{cases} \ 0 \hspace{1cm} \textrm{if} \quad y\in T^{S,\infty}_{\{x_n\}}\\
\ 1 \hspace{1cm} \textrm{if} \quad y \not \in T^{S,\infty}_{\{x_n\}}.
\end{cases}
\end{align*}
\end{lem}
\begin{proof}
If $y\in T^{S,\infty}_{\{x_n\}}$, then  $\lambda_y(x_n) = 0$ for all $n$. If $y \not \in T^{S,\infty}_{\{x_n\}}$, then $\lambda_{y}(x_n)$ and $(x_n \vee 1)^y$ are polynomials with the same degree and the same leading coefficient (of 1).
\end{proof}

The following corollary is used throughout.

\begin{cor}\label{cor23}
Let $\{x_n\}$  be a tier-sequence of a reaction network $(\S,\C,\Re)$. Suppose $y \not \in T^{S,\infty}_{\{x_n\}}$ and $y \in T^{D,1}_{\{x_n\}}$. Then $y \in \Ts1$ and $\lmt \lambda_y(x_n)=\infty$.
\end{cor}
\begin{proof} 
Let $y' \not \in T^{S,\infty}_{\{x_n\}}$. Note that
\[
 \frac{\lambda_{y'}(x_n)}{\lambda_{y}(x_n)} =  \frac{\lambda_{y'}(x_n)}{(x_n \vee 1)^{y'}}\cdot \frac{(x_n \vee 1)^{y'}}{(x_n\vee 1)^{y}}\cdot \frac{(x_n \vee 1)^{y}}{\lambda_{y}(x_n)}. 
\]
By Lemma \ref{lem22}, as $n\to \infty$ the first and third terms on the right of the above equation converge to 1.  Therefore, 
\[
\lim_{n\rightarrow \infty} \frac{\lambda_{y'}(x_n)}{\lambda_{y}(x_n)} = \lim_{n\to \infty}\frac{(x_n \vee 1)^{y'}}{(x_n\vee 1)^{y}}.
\]
Because $y\in T^{D,1}_{\{x_n\}}$, the last limit is either 0 (if $y' \notin \Td1$) or some $C>0$ (if $y' \in \Td1$).  Thus, $y'$ cannot be in a higher tier than $y$ in the S-type partition.  Hence,  $y \in \Ts1$. Moreover, since $y\not \in \Tsinf$, by \lemref{lem : Td1 goes infinity} and \lemref{lem22}, we have $\lmt \lambda_{y}(x_n)=\infty$. 
\end{proof}

\section{Tier structures for positive recurrence}
\label{sec:main}

In this section, we  provide a  theorem that provides sufficient conditions on  D-type and S-type tiers for the Markov process associated to a reaction network to be stable in the sense of Theorem \ref{thm11}.  We require the following lemma.

\begin{lem}\label{lemma:main}
Let $\mathcal A$ be the generator \eqref{gen5} of the continuous time Markov chain associated to a reaction network $(\S,\C,\Re)$ with mass-action kinetics \eqref{mass} and rate constants $\kappa_{y\to y'}$. Let $V$ be the function defined  in and around \eqref{eq:MainLyapunov}.   For each tier sequence $\{x_n\}$ 
there is a constant $C>0$  for which
\begin{eqnarray}\label{asymptotic of AV} \A V(x_n) \le \sum_{y\rightarrow y'\in \Re}\kappa_{y\rightarrow y'} \lambda_{y}(x_n)\Big ( \ln{\left(\frac{(x_n \vee 1)^{y'}}{(x_n \vee 1)^y}\right) +C \Big )}.
\end{eqnarray}
 \end{lem}
\begin{proof}
Let $I	= \{ i \ |\  x_{n,i} \rightarrow \infty,  \ \ \textrm{as} \ \ n \rightarrow \infty \} \neq \emptyset$.
Then for a reaction $y\rightarrow y' \in \Re$, there exists $C_{y\to y'}>0$ such that
\begin{align}
V&(x_n+y'-y)-V(x_n)\notag \\ 
&\le \dsum_{i\in I}[ (x_{n,i}+y'_i-y_i)(\ln{(x_{n,i}+y'_i-y_i)}-1) - x_{n,i}(\ln{(x_{n,i})}-1)] + C_{y\to y'}\notag\\
&=  \dsum_{i\in I}[ x_{n,i}\ln \left(1+\frac{y'_i-y_i}{x_{n,i}}\right)+y_i-y'_i + (y'_i-y_i)\ln{(x_{n,i}+y'_i-y_i)}] + C_{y\to y'}\notag,
\end{align}
where we simply grouped those terms not going to infinity into the constant.  Using the fact that $\lim_{t \to \infty} (1 + \tfrac{\alpha}{t})^t = e^\alpha$ for any $\alpha$, we have  that
\[
x_{n,i}\ln \left(1+\frac{y'_i-y_i}{x_{n,i}}\right)+y_i-y'_i \rightarrow 0, \ \ \ \textrm{as} \ \ n \rightarrow \infty,
\]
 for each $i\in I$.
Hence, there are $C_{y\to y'}'>0$ for which 
\[
V(x_n+y'-y)-V(x_n) \le \dsum_{i \in I} (y'_i-y_i)\ln (x_{n,i}+y'_i-y_i) + C_{y\to y'}',
\]
for each $n$.  Noting that 
\[
	\ln (x_{n,i}+y'_i-y_i)^{y'_i-y_i} - \ln(x_{n,i}^{y_i'-y_i}) \to 0, 
\]
as $n\to \infty$ implies the existence of a $C_{y\to y'}''>0$ and a $C_{y\to y'}'''>0$ for which 
\begin{align*}
V(x_n+y'-y)&-V(x_n)   \le \dsum_{i \in I}\ln \left(  x_{n,i}^{y'_i-y_i}\right) + C_{y\to y'}'' = \ln \left(  \dprod_{i\in I} \frac{x_{n,i}^{y'_i}}{x_{n,i}^{y_i}}\right) + C_{y\to y'}''\\
&\le   \ln{\left(\dprod_{i=1}^d \frac{(x_{n,i} \vee 1)^{y'_i}}{(x_{n,i} \vee 1)^{y_i}}\right)} + C_{y\to y'}''' = \ln{\left( \frac{(x_n \vee 1)^{y'}}{(x_n \vee 1)^y}\right)} +C_{y\to y'}'''.
\end{align*}
Hence, 
\begin{align*}
\mathcal{A}V(x_n) &= \sum_{y\to y'\in \Re} \kappa_{y\to y'} \lambda_y(x_n) (V(x_n+y'-y) - V(x_n))\\
&\le   \sum_{y\to y'\in \Re} \kappa_{y\to y'} \lambda_y(x_n) \left( \ln{\left( \frac{(x_n \vee 1)^{y'}}{(x_n \vee 1)^y}\right)} +C_{y\to y'}'''\right).
\end{align*}
The proof is completed by taking $C = \max_{y\to y'\in \Re}\{ C_{y\to y'}'''\}$.
\end{proof}
 
\bigskip 
 
 The following is our main analytic theorem related to tiers.
\begin{thm}\label{thm32}
Let $(\S,\C,\Re)$ be a reaction network. Suppose that
 \begin{align}\label{hyp:main}
T^{S,1}_{\{x_n\}} \cap D_{\{x_n\}}   \neq \emptyset.
\end{align}
for any tier-sequence $\{x_n\}$, where   $T^{S,i}_{\{x_n\}}$ are the S-type tiers and  $D_{\{x_n\}}$ is the set of source complexes for the descending reactions along $\{x_n\}$.  
Then for any choice of rate constants  the Markov process  with intensity functions \eqref{mass} associated to the reaction network $(\S,\C,\Re)$ satisfies the following: each state in a closed, irreducible component of the state space is positive recurrent; moreover, if  $\tau_{x_0}$ is the time for the process to enter the union of the closed irreducible components given an initial condition  $x_0$, then $\mathbb{E}[\tau_{x_0}] < \infty$.
\end{thm} 

\begin{proof}
We will show there exists a finite set $K \subset \Z^d_{\ge 0}$ such that $\A V(x) \le -1$ for all $x \in K^c$.  An application of Theorem \ref{thm11}  then completes the proof.  We proceed with an argument by contradiction. For ease of notation, the positive constant $C$ appearing in different lines may vary. 

Suppose, in order to find a contradiction, that there exists a sequence $\{ x_n \}$ with $\lim_{n\rightarrow \infty}|x_n| =\infty$ and $\mathcal{A}V(x_n) > -1$ for all $n$. By Lemma \ref{lem21}, there exists a subsequence which is a tier-sequence.  For simplicity, we also denote this tier-sequence by $\{x_n\}$.  Denote the S-type tiers, D-type tiers, and source complexes for the descending reactions for this particular tier-sequence by $T^{S,i}_{\{x_n\}}, T^{D,i}_{\{x_n\}},$ and $\Des$, respectively.
Our main hypothesis \eqref{hyp:main} implies that there exists a reaction $y_0\rightarrow y'_0$ such that $y_0 \in \Td1 \cap T^{S,1}_{\{x_n\}}$ and $y_0'\in \Tdi$ for some  $i>1$. 
Starting with an application of Lemma \ref{lemma:main}, we have the existence of a 
$C>0$ for which
\begin{align}
 \mathcal{A}V(x_n) \nonumber &\le  \sum_{\substack{y\rightarrow y' \in \Re }} \kappa_{y\rightarrow y'}\lambda_{y}(x_n)\Big(\ln{\left( \frac{(x_n\vee 1)^{y'}}{(x_n\vee 1)^y} \right)}+C \Big) \notag\\
&=  \lambda_{y_0}(x_n) \Bigg( \underbrace{ \dsum_{\substack{y\rightarrow y' \in \Re \\ y'\succ_D y }} \frac{\kappa_{y\rightarrow y'}\lambda_{y}(x_n)}{\lambda_{y_0}(x_n)}\ln{\left(\frac{(x_n\vee 1)^{y'}}{(x_n\vee 1)^y}\right)}}_{I}
+ \underbrace{\dsum_{\substack{y\rightarrow y' \in \Re  \\ y \succ_D y' }} \frac{\kappa_{y\rightarrow y'}\lambda_{y}(x_n)}{\lambda_{y_0}(x_n)}\ln{\left(\frac{(x_n\vee 1)^{y'}}{(x_n\vee 1)^y}\right)}}_{II} \notag \\
&\hspace{.2in} +\underbrace{\dsum_{\substack{y\rightarrow y' \in \Re\\ y \sim_D y' }} \frac{\kappa_{y\rightarrow y'}\lambda_{y}(x_n)}{\lambda_{y_0}(x_n)}\ln{\left(\frac{(x_n\vee 1)^{y'}}{(x_n\vee 1)^y}\right)}}_{III} + 
\underbrace{\sum_{\substack{y\rightarrow y' \in \Re \vspace{0.3cm}}}\frac{\kappa_{y\rightarrow y'}\lambda_{y}(x_n)}{\lambda_{y_0}(x_n)}C}_{IV} \Bigg ). \label{eq:main eq for AV}
\end{align}
First note that by Corollary \ref{cor23}, we know $\lambda_{y_0}(x_n) \to \infty$ as $n \to \infty$.    To conclude the proof, we will show that term $II$ will converge to $-\infty$, as $n \to \infty$, and that all the other terms remain uniformly bounded in $n$.  One fact we will use repeatedly is the following:  because  $y_0 \in \Td1 \cap T^{S,1}_{\{x_n\}}$, there exists a constant $C'>0$ such that for all complexes $y\in \C$ and all $n$ large enough
\begin{align}\label{maintheorem1}
\frac{\lambda_y(x_n)}{\lambda_{y_0}(x_n)} < C' \hspace{1cm} \textrm{and} \hspace{1cm} \frac{(x_n\vee 1)^{y}}{(x_n\vee 1)^{y_0}} < C'.
\end{align}
Note that $\eqref{maintheorem1}$ immediately implies that terms $III$ and $IV$ are uniformly bounded in $n$.  

We turn to term $I$.  By adding and subtracting appropriate log terms,
\begin{align}
I &= \sum_{\substack{y\rightarrow y' \in \Re \\ y'\succ_D y }} \frac{\kappa_{y\rightarrow y'}\lambda_y(x_n)}{\lambda_{y_0}(x_n)}\ln{\left(\frac{(x_n\vee 1)^{y_0}}{(x_n\vee 1)^y}\right)} + \sum_{\substack{y\rightarrow y' \in \Re \\ y'\succ_D y }} \frac{\kappa_{y\rightarrow y'}\lambda_y(x_n)}{\lambda_{y_0}(x_n)}\ln{\left(\frac{(x_n\vee 1)^{y'}}{(x_n\vee 1)^{y_0}}\right)} \notag \\ 
&= \sum_{\substack{y\rightarrow y' \in \Re \\ y'\succ_D y }} \left( \frac{\kappa_{y\rightarrow y'}\lambda_y(x_n)}{\lambda_{y_0}(x_n)}\frac{(x_n\vee 1)^{y_0}}{(x_n \vee 1)^{y}} \right)\frac{(x_n\vee 1)^y}{(x_n \vee 1)^{y_0}}\ln{\left(\frac{(x_n\vee 1)^{y_0}}{(x_n\vee 1)^y}\right)} \label{maintheorem11}\\
&\hspace{.25in}+ \sum_{\substack{y\rightarrow y' \in \Re \\ y'\succ_D y }} \frac{\kappa_{y\rightarrow y'}\lambda_y(x_n)}{\lambda_{y_0}(x_n)}\ln{\left(\frac{(x_n\vee 1)^{y'}}{(x_n\vee 1)^{y_0}}\right)}\label{maintheorem12} 
\end{align}
By Lemma \ref{lem22}, a part of  term \eqref{maintheorem11} can be shown to be bounded:
\[ \lim_{n\rightarrow \infty} \frac{\kappa_{y\rightarrow y'}\lambda_y(x_n)}{\lambda_{y_0}(x_n)}\frac{(x_n\vee 1)^{y_0}}{(x_n \vee 1)^{y}} = \begin{cases}
\kappa_{y\rightarrow y'}, & \text{if} \ y\not \in \Tsinf\\
0, & \text{if} \ y\in \Tsinf.
\end{cases}
 \]
  In addition, since $y_0 \in T^{D,1}_{\{x_n\}}$ and $y \not \in T^{D,1}_{\{x_n\}}$,  
\[ 
\lim_{n\rightarrow \infty} \frac{(x_n\vee 1)^{y}}{(x_n\vee 1)^{y_0}} \ln \left({\frac{(x_n\vee 1)^{y_0}}{(x_n\vee 1)^{y}}}\right)=0,
\] 
where we are utilizing  $\lim_{t \to 0^+} t \ln(1/t) = 0$.  We conclude that the term \eqref{maintheorem11} converges to zero as $n\to \infty$.   Finally, \eqref{maintheorem1} shows that the term \eqref{maintheorem12} is uniformly bounded in $n$.

We now turn to showing that term $II$ converges to $-\infty$, as $n\to \infty$.  We have
\ \begin{align}
II =& \sum_{\substack{y\rightarrow y' \in \Re \\ y \succ_D y' }} \frac{\kappa_{y\rightarrow y'}\lambda_y(x_n)}{\lambda_{y_0}(x_n)}\ln{\left(\frac{(x_n\vee 1)^{y'}}{(x_n\vee 1)^y}\right)} \notag\\
=&  \frac{\kappa_{y_0\rightarrow y_0'}\lambda_{y_0}(x_n)}{\lambda_{y_0}(x_n)}\ln{\left(\frac{(x_n\vee 1)^{y'_0}}{(x_n\vee 1)^{y_0}}\right)}+ \dsum_{\substack{y\rightarrow y' \in \Re\setminus \{y_{_0} \rightarrow y_0'\} \\ y\succ_D y'}} \frac{\kappa_{y\rightarrow y'}\lambda_{y}(x_n)}{\lambda_{y_0}(x_n)}\ln{\left(\frac{(x_n\vee 1)^{y'}}{(x_n\vee 1)^y}\right)} \label{maintheorem21},
\end{align}
where we recall that $y_0'$ is the product complex of the descending reaction $y_0 \to y_0'$.  Note that the first term on the right of \eqref{maintheorem21}, converges to $-\infty$ since $\kappa_{y_0\to y_0'} > 0$ and
\begin{align*}
\lmt\ln{\left(\frac{(x_n\vee 1)^{y_0'}}{(x_n\vee 1)^{y_0}}\right)} =-\infty.
\end{align*} 
The second term on the right of \eqref{maintheorem21} may be an empty sum.  However, if there are any terms in the sum, they must be less than or equal to zero for $n$ large enough.     Indeed, this follows because  
\[
  \lmt\ln{\left(\frac{(x_n\vee 1)^{y'}}{(x_n\vee 1)^{y}}\right)} =-\infty,
\]
and $\lambda_y(x_n)\ge 0$.
 Therefore, we conclude that the term $II$ converges to $-\infty,$ as $n \rightarrow \infty$.

Hence, we must conclude that  $\lim_{n\rightarrow \infty} \mathcal{A}V(x_n) = -\infty $. However, this is in contradiction to the assumption that we made at the beginning of this proof, and the result is shown. 
\end{proof}

We have following corollary of the \thmref{thm32}.

\begin{cor}\label{cor33}Let $(\S,\C,\Re)$ be a reaction network. Suppose that 
\begin{align*}
\Des \neq \emptyset \quad \text{and} \quad \Td1=\Ts1,
\end{align*}
for any tier-sequence $\{x_n\}$.
Then for any choice of rate constants  the Markov process with intensity functions \eqref{mass} associated to the reaction network $(\S,\C,\Re)$ satisfies the following: each state in a closed, irreducible component of the state space is positive recurrent; moreover, if  $\tau_{x_0}$ is the time for the process to enter the union of the closed irreducible components given an initial condition  $x_0$, then $\mathbb{E}[\tau_{x_0}] < \infty$. 
\end{cor}

\section{Network structures that guarantee positive recurrence}
\label{sec:applications}

With Theorem \ref{thm32} and Corollary \ref{cor33} in hand, we turn to proving our main results: Theorems \ref{thm:main1} and \ref{thm52}. 

\subsection{The single linkage class case}
In the papers  \cite{AndGAC_ONE2011, AndBounded_ONE2011} it was shown that deterministic models of reaction networks with a single weakly reversible linkage class were persistent and bounded.  In this section, we generalize this result to the stochastic setting with the additional assumption that all in-flows and out-flows are present.
 We begin with a lemma.

\begin{lem}\label{lem for single tier1}
Let $(\S,\C,\Re)$ be a reaction network with $\S=\{S_1,S_2,\cdots,S_d\}$. Suppose $S_i \to \emptyset \in \Re$ for some species $S_i$. For a tier sequence $\{x_n\}$, if $S_i \in \Td1$ , then 
\begin{align}\label{suffcondition}
S_i \in T^{S,1}_{\{x_n\}} \cap D_{\{x_n\}}.
\end{align}  
\end{lem}
\begin{proof}
Note that Lemma \ref{lem : Td1 goes infinity} implies $\emptyset \notin \Td1$.
Hence, $S_i \in D_{\{x_n\}}$. Because $\lambda_{S_i}(x_n) \ne 0$,  Corollary \ref{cor23} implies $S_i \in T^{S,1}_{\{x_n\}}$.
%
%
%
%
%
 \end{proof}

We now prove  Theorem \ref{thm:main1}, which we restate here.

\begin{customthm}{1}
\textit{Let $(\S,\C, \Re)$ be a weakly reversible, binary reaction network that has a single linkage class.   Let $\widetilde \Re =  \Re \cup_{S\in \S} \{\emptyset \to S, S \to \emptyset\}$ and $\widetilde{\C} = \C \cup\{\emptyset\}\cup\{S \ | \ S \in \S\}$. Then, for any choice of rate constants, every state of the Markov process with intensity functions \eqref{mass} associated to the reaction network $(\S,\widetilde{\C},\widetilde \Re)$ is positive recurrent.  
}
\end{customthm}


\begin{proof}
For concreteness, order the species as $\S = \{S_1,\dots,S_d\}$.

First suppose $\C$ consists of either only binary complexes or only unary complexes. Then $ (y'-y) \cdot \vec{1}=0$ for all $y \to y' \in \Re$, where $\vec{1}=(1,1,\dots,1)\in \Z^d$. Let $\A$ be a generator of the Markov process associated to the reaction network $(\S,\widetilde \C, \widetilde \Re)$. Then for the function $W(x)=x_1+x_2+\cdots+x_d$,  we have
\begin{align*}
\mathcal{A}W(x) 
&= - \sum_{i=1}^d\kappa_{S_i \rightarrow \emptyset}x_{i} + \sum_{i=1}^d \kappa_{\emptyset \rightarrow S_i}
\end{align*}
 Thus, for an arbitrary sequence $\{x_n\} \in \Z^d_{\ge 0}$ such that $|x_n|\rightarrow \infty$, as $n\rightarrow \infty$, 
\begin{align*}
\mathcal{A}W(x_n) \rightarrow -\infty.
\end{align*}
This implies $\mathcal{A}W(x) < -1$ for all $x$ but finitely many.  Hence, we may apply Theorem \ref{thm11}.  Noting that weak reversibility implies that all states  are contained within a closed, irreducible component of the state space  then finishes the proof.

Now we suppose $\C$ does not contain only binary complexes or only unary complexes. Let $\{x_n\}$ be a tier-sequence. We will show that \eqref{hyp:main} holds for the expanded network $(\S,\widetilde{\C},\widetilde \Re)$, in which case an application of Theorem \ref{thm32} is applicable.  Noting that weak reversibility implies that all states  are contained within a closed, irreducible component of the state space  then finishes the proof. 

There are three cases to consider.

\vspace{.1in}
\noindent \textbf{Case 1.}  Assume that all complexes in $\C$ are in $\Td1$.\\

 If there is no unary complex in $\mathcal{C}$, then we must have $\emptyset\in \C$ (since not all complexes are binary). However, by Lemma \ref{lem : Td1 goes infinity} we know $\emptyset \notin \Td1$.  Since this would contradict that all complexes in $\C$ are in $\Td1$, it must be that at least  one unary complex is in $\C$. Since a unary complex is in $\Td1$, \lemref{lem for single tier1} implies $D_{\{x_n\}} \cap \Ts1 \neq \emptyset$.

\vspace{.1in}
\noindent \textbf{Case 2.} Assume that some of the complexes in $\C$ are not in $\Td1$, and one complex in $D_{\{x_n\}}$ is binary.\\

Since (i) not all complexes in $\C$ are in $\Td1$, and (ii) $(\S,\C,\Re)$ is weakly reversible, there is a reaction $y_0 \rightarrow y_0'$ such that $y_0\in \Td1$ and  $y'_0 \not \in \Td1$. We assume  that $y_0=S_i+S_j$ for some $i,j \in \{1,2,\dots ,d\}$ (where we allow $i= j$). 

If $y_0 \in \Ts1$, then we may conclude the proof by an application of  Theorem \ref{thm32}.  Hence, we  assume that  $y_0 \not \in \Ts1$, and must demonstrate the existence of a descending reaction $y\to y'$ such that $y \in \Ts1$. 

 By  Corollary \ref{cor23}, we must have $y_0 \in T^{S,\infty}_{\{x_n\}}$. This means $i\neq j$, and, without loss of generality, $x_{n,j}=0$ for all $n$.  We further conclude from Lemma \ref{lem : Td1 goes infinity} that $x_{n,i}\to \infty$ as $n \to \infty$.
Also, since $S_i +S_j\in \Td1$, it must be that $S_i \in \Td1$ since when $x_{n,j} = 0$, we have
\[
	(x_n \vee 1)^{S_i} = x_{n,i} = x_{n,i} \cdot (x_{n,j} \vee 1) =(x_n\vee 1)^{S_i+S_j}.
\]
An application of \lemref{lem for single tier1} then completes the argument.


\vspace{.1in}
\noindent \textbf{Case 3.} Assume that some of the complexes in $\C$ are not in $\Td1$, and one complex in $D_{\{x_n\}}$ is unary.\\

An application of \lemref{lem for single tier1}  completes the argument.
\end{proof}

Consider the following example of substrate-enzyme kinetics.
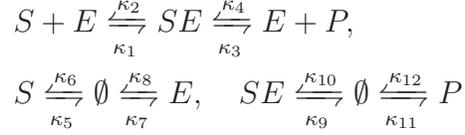
\begin{example}
\begin{align*}
&S+E \xleftrightharpoons[\kappa_1]{\kappa_2} SE \xleftrightharpoons[\kappa_3]{\kappa_4} E+P,\\ & S \xleftrightharpoons[\kappa_5]{\kappa_6} \emptyset \xleftrightharpoons[\kappa_7]{\kappa_8} E,   \quad
SE \xleftrightharpoons[\kappa_9]{\kappa_{10}} \emptyset \xleftrightharpoons[\kappa_{11}]{\kappa_{12}} P 
\end{align*} 
This reaction network consists of a single linkage class which is weakly reversible (the top linkage class), and in-flows and out-flows for all species.  Moreover, the state space $\mathbb{S} = \Z^{4}_{\ge 0}$ is irreducible.    Therefore the associated Markov process for this reaction network is positive recurrent for any choice of rate constants $\kappa_1,\kappa_2,\dots,\kappa_{12}$. \hfill $\triangle$
\end{example}

\subsection{Double-full binary reaction networks}
In this section, we prove Theorem \ref{thm52}.  We begin with a necessary lemma that captures the usefulness of the double-full assumption.

\begin{lem}\label{lem51}
Let $(\S, \C, \Re)$ be a double-full, binary reaction network with $\mathcal{S} =\{ S_1,S_2,\dots, S_d\}$. Let $\{x_n\}$ be a tier-sequence of $(\S, \C, \Re)$. Then the following holds:
\begin{enumerate}
\item If $S_i+S_j \in T^{D,1}_{\{x_n\}}$, then $2S_i, 2S_j \in T^{D,1}_{\{x_n\}}$. Thus, $\dlim_{n\rightarrow \infty} \frac{(x_{n,i}\vee 1)}{(x_{n,j}\vee 1)} = C$ for some constant $C>0$.
\item $T^{D,1}_{\{x_n\}} \subset \{ S_i + S_j \ | \ i,j =1,2,\dots, d\}$ and $2S_i \in \Td1$ for some $i=1,2\dots,d$ .
That is, $T^{D,1}_{\{x_n\}}$ consists of only binary complexes and always contains at least one double complex.
\item $T^{D,1}_{\{x_n\}} = T^{S,1}_{\{x_n\}}$.
\end{enumerate}
\end{lem}
\begin{proof} 
Let $I=\{i \ | \ \lmt x_{n,i} = \infty  \}$.

For the first claim, if $i=j$, then the result is trivial. Thus, let $i \neq j$. Suppose $2S_i \not \in T^{D,1}_{\{x_n\}}$. Then
\begin{align*}
\lmt \frac{(x_n \vee 1)^{2S_i}}{(x_n\vee 1)^{S_i+S_j}}
=\lmt \frac{(x_{n,i} \vee 1)}{(x_{n,j}\vee 1)} = 0.
\end{align*}
Hence, 
\begin{align*}
\dlim_{n\rightarrow \infty} \frac{(x_n \vee 1)^{S_i + S_j}}{(x_n\vee 1)^{2S_j}}
=\dlim_{n\rightarrow \infty} \frac{(x_{n,i} \vee 1)}{(x_{n,j}\vee 1)} = 0.
\end{align*}
This implies $2S_j \succ_D S_i+S_j$ which is in contradiction to the assumption $S_i+S_j \in \Td1$. Therefore, $2S_i \in T^{D,1}_{\{x_n\}}$. In same way, we can show $2S_j \in T^{D,1}_{\{x_n\}}$.

We turn to the second claim. We will show that unary complexes and the zero complex cannot be in $\Td1$. First $\emptyset \not \in \Td1$ follows by \lemref{lem : Td1 goes infinity}. Suppose now that $S_m \in \Td1$ for some $m$. Then either $2S_m \succ_D S_m$ or $2S_k   \succ_D S_m$ for some $k\in I$, because  
\begin{align*}
 &\lmt \frac{(x_n\vee 1)^{S_m}}{(x_n \vee 1)^{2S_m}} = \lmt \frac{1}{x_{n,m}} = 0 \quad \text{if} \ \ m \in I \quad \text{and}, \\ 
&\lmt \frac{(x_n\vee 1)^{S_m}}{(x_n \vee 1)^{2S_k}} = \lmt \frac{x_{n,m}}{x_{n,k}^2} = 0 \quad \text{if} \ \ m \not \in I.
\end{align*}
Thus $S_m \not \in \Td1$. Part 1 shows that there is at least one $i$ for which $2S_i\in \Td1$.

For the last claim, we will first show that $\Ts1 \subset \Td1$. Let $y \in \Ts1$. By result 2, $2S_i \in \Td1$ for some $i$. Note that, by Corollary \ref{cor23}, $2S_i \in \Ts1$ since \lemref{lem : Td1 goes infinity} implies that $\lambda_{2S_i}(x_n)\neq 0$ for large $n$. Applying  \lemref{lem22}, we have \[\lmt \frac{(x_n\vee 1)^{2S_i}}{(x_n\vee 1)^y} = \lmt\frac{\lambda_{2S_i}(x_n)}{\lambda_y(x_n)} = C \quad \text{for some constant $C>0$}.\]
This means that $y \sim_D 2S_i$. Therefore $y \in \Td1$ and we conclude that $\Ts1 \subset \Td1$. Now we will show $\Td1 \subset \Ts1$. Let $y \in \Td1$. It is sufficient to show that $y \not \in \Tsinf$ by Corollary \ref{cor23}. By result 2, $y=S_i +S_j$ for some $i$ and $j$ (where we allow $i\neq j$), and $2S_i \in \Td1$ and $2S_j \in \Td1$. By \lemref{lem : Td1 goes infinity}, $x_{n,i}\to \infty$ and $x_{n,j}\to \infty$, as $n\to \infty$.  Therefore $\lambda_{y}(x_n) \neq 0$ for large $n$ and hence $y \not \in \Tsinf$.
\end{proof}

\lemref{lem51} concludes that $T^{D,1}_{\{x_n\}}$ and $T^{S,1}_{\{x_n\}}$ for a double-full, binary reaction network are always equal and consist of binary complexes. Now we restate the second main result \thmref{thm52} with its proof.

\begin{customthm}{2}
Let $(\S,\C, \Re)$ be a binary reaction network satisfying the following two conditions:
\begin{enumerate}
\item the reaction network is double-full, and
\item for each double complex (of the form $2S_i$) there is a directed path within the reaction graph beginning with the double complex itself and ending with either a unary complex (of the form $S_j$) or  the zero complex.
\end{enumerate} 
Then, for any choice of rate constants,  the Markov process with intensity functions \eqref{mass} associated to the reaction network $(\S,\C,\Re)$ satisfies the following: each state in a closed, irreducible component of the state space is positive recurrent; moreover, if  $\tau_{x_0}$ is the time for the process to enter the union of the closed irreducible components given an initial condition  $x_0$, then $\mathbb{E}[\tau_{x_0}] < \infty$.
\end{customthm}
\begin{proof}
Let $\{x_n\}$ be a tier-sequence. Result 3 in Lemma \ref{lem51} shows $T^{D,1}_{\{x_n\}} = T^{S,1}_{\{x_n\}}$.  Thus, so long as a descending reaction can be shown to exist, an application of Corollary \ref{cor33} will complete the proof.

 By result 2 in Lemma \ref{lem51}, there exists a double complex $2S_i \in T^{D,1}_{\{x_n\}}$ for some $i$. By our hypothesis, there exists a directed path from $2S_i$ to a unary or the zero complex $y'$ in the reaction graph. According to result 2 in Lemma \ref{lem51}, $y' \not \in \Td1$. Therefore a descending reaction exists within the directed path from $2S_i$ to $y'$.
\end{proof}

We demonstrate    \thmref{thm52} with an example.

\begin{example}
The following reaction network contains $5$ species, $14$ complexes and $14$ reactions. 
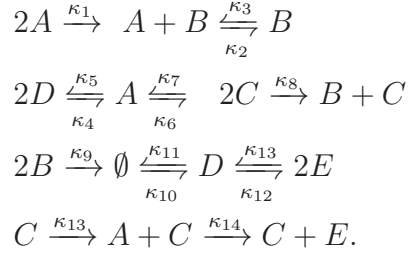
\begin{align*}\label{ex}
& 2A \xrightarrow{\kappa_1} \  A+B \xleftrightharpoons[\kappa_2]{\kappa_3} B   \\
& 2D \xleftrightharpoons[\kappa_4]{\kappa_5}  A\xleftrightharpoons[\kappa_6]{\kappa_7} \  \ 2C \xrightarrow{\kappa_8} B+C  \\
&2B \xrightarrow{\kappa_9} \emptyset \xleftrightharpoons[\kappa_{10}]{\kappa_{11}} D \xleftrightharpoons[\kappa_{12}]{\kappa_{13}} 2E \\
& C \xrightarrow{\kappa_{13}} A+C \xrightarrow{\kappa_{14}} C+E. 
\end{align*}
This binary reaction network is double-full. Moreover, for each double complex ($2A$, $2B$, $2C$, $2D$, $2E$ and $2F$), there is a directed path within the reaction graph beginning with the double complex itself and ending with either a unary complex or the zero complex. Therefore the conditions in  Theorem \ref{thm52} hold. Since $\mathbb{S} = \Z^5_{\ge 0}$ is irreducible for this model, the associated continuous time Markov chain is positive recurrent regardless of choice of the rate constants $\kappa_1,\dots,\kappa_{14}$.\hfill $\triangle$
\end{example}

\subsection{More results on double-full, binary reaction networks}
In this section, we  provide classes of double-full, binary reaction networks for which condition 2 of Theorem \ref{thm52} (the ``path condition'') does not hold, but for which the conclusions of Theorem \ref{thm52} still hold.

We begin with a technical lemma.

\begin{lem}\label{lem53}
Let $(\S,\C,\Re)$ be a double-full, binary reaction network. Suppose the following:
\begin{enumerate}
\item  $\L$ is a weakly reversible linkage class with $S,\tilde{S} \in \S(\L)$ (where we allow $S = \tilde{S}$) such that $S+\tilde{S} \in \C$.
\item There is a directed path within the reaction graph beginning with $S+\tilde{S}$ and ending with a unary  or the zero complex.
\end{enumerate}
Then for any tier-sequence $\{x_n\}$ the following holds: if there is a complex $y$ in the linkage class $\L$ (i.e. $y\in \C(\L)$) that is in $\Td1$, then $\Des \neq \emptyset$.
\end{lem}

We demonstrate the lemma with an example.

\begin{example}
Consider the following reaction network
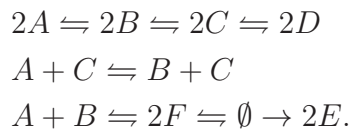
\begin{align*}
	&2A \leftrightharpoons 2B \leftrightharpoons 2C \leftrightharpoons 2D\\
	&A + C \leftrightharpoons B + C\\
	&A + B \leftrightharpoons  2F \leftrightharpoons \emptyset \to 2E.
\end{align*}
Let $\L$ be the middle linkage class (i.e., $A + C \leftrightharpoons B + C$).  Then, $A ,B \in \S(\L)$, and there is a directed path from $A + B$ to $\emptyset$, showing that conditions 1 and 2 are met.  Hence, we conclude that if $\{x_n\}$ is a tier-sequence and either $A + C$ or $B+C$ are in $\Td1$, then we necessarily have that $D_{\{x_n\}} \ne 0$.  In this case, the descending reaction could be in any of the three linkage classes (depending upon the particular sequence $\{x_n\}$). \hfill $\triangle$
\end{example}

\begin{proof} 
Let $\S=\{S_1,S_2,\dots,S_d\}$.  Assume that $\{x_n\}$ is a tier sequence and that there is a  $y \in \C(\mathcal L)$ such that $y \in \Td1$.  We must show that $D_{\{x_n\}} \ne \emptyset$.

\vspace{.1in}

\noindent \textbf{Case 1.}  If there is a complex $y' \in C(\mathcal{L})$ that is not in $\Td1$, then there necessarily  exits a descending reaction along $\{x_n\}$ by the weak reversibility of $\L$. 

\vspace{.1in}

\noindent \textbf{Case 2.}
Now suppose that all complexes in $C(\mathcal{L})$ are in $T^{D,1}_{\{x_n\}}$.
We  will show $S+\tilde{S} \in \Td1$. We assume  $S=S_i$ and $\tilde{S}=S_j$ for some $i$ and $j$ (where, again, we could have $i = j$).
 By  result 2 in \lemref{lem51}, $\C(\L)$ contains only binary complexes. Thus $S_i+S_m \in \C(\L)$ for some $m$. Indices $i,j$ and $m$ are not necessarily all distinct. By  result 1 in Lemma \ref{lem51}, $\{2S_i,2S_m,2S_j\} \subset \Td1$. Therefore
\[ \lmt \frac{(x_n\vee 1)^{S_i+S_j}}{(x_n\vee 1)^{S_i+S_m}} = \lmt \frac{(x_{n,j}\vee 1)}{(x_{n,m}\vee 1)} =\sqrt{\lmt \frac{(x_{n}\vee 1)^{2S_j}}{(x_{n}\vee 1)^{2S_m}} } = C \quad \text{for some $C>0$}.\]
Therefore $S_i+S_j \in \Td1$. By  hypothesis 2, there exists a directed path from $S_i+S_j$ to a unary complex or the zero complex within the reaction graph. Since only binary complexes can be in $\Td1$ in a  double-full reaction network, there exists a descending reaction along $\{x_n\}$ within the directed path.
\end{proof}

\begin{thm}\label{thm53}
Let $(\S,\C,\Re)$ be a double-full, binary reaction network with linkage classes $\L_1,\L_2,\dots,\L_\ell$. Suppose there is an integer $m \in \{1,2,\dots,\ell-1\}$ such that:
\begin{enumerate}
\item For $i \le m$, $\L_i$ is weakly reversible and $\C(\L_i)$ contains only binary complexes.
\item For each $i\le m$, there exists $S,\tilde{S} \in \S(\mathcal{L}_i)$ (where we allow $S \neq \tilde{S}$) such that 
\begin{enumerate}
\item[(i)]   $S+\tilde{S} \in \C$,
\item[(ii)] there is a directed path within the reaction graph beginning with $S+\tilde{S}$ and ending with a unary  or the zero complex.
\end{enumerate}
\item For each double complex $2S$, either $2S \in \C(\L_i)$ for some $i  \le m$ or there is a directed path within the reaction graph beginning with $2S$ and ending with a unary complex or the zero complex.
\end{enumerate}
Then, for any choice of rate constants, the Markov process with intensity function \eqref{mass} associated to the reaction network $(\S,\C,\Re)$ satisfies the following: each state in a closed, irreducible component of the state space is positive recurrent; moreover, if  $\tau_{x_0}$ is the time for the process to enter the union of the closed irreducible components given an initial condition  $x_0$, then $\mathbb{E}[\tau_{x_0}] < \infty$.
\end{thm}
 
\begin{proof}
Let $\{x_n\}$ be a tier-sequence. Result 3 in Lemma \ref{lem51} shows $T^{D,1}_{\{x_n\}} = T^{S,1}_{\{x_n\}}$.  Thus, so long as a descending reaction can be shown to exist, an application of Corollary \ref{cor33} will complete the proof.

By  result 2 in Lemma \ref{lem51}, $2S_j \in T^{D,1}_{\{x_n\}}$ for some $j$. If there is a directed path beginning with $2S_j$ and ending with a unary complex or the zero complex within the reaction graph, we have a descending reaction along $\{x_n\}$ within the directed path. Otherwise, $2S_j \in \C(\L_i)$ for some $i \le m$ by the hypothesis, and hence $\Des \neq \emptyset$ by \lemref{lem53}.
\end{proof} 

We demonstrate Theorem \ref{thm53} with an example.

\begin{example}
Consider the following, which is a double-full, binary reaction network for which the conditions in the \thmref{thm53} hold.
\begin{align*}
&A+B\rightleftharpoons 2B \rightleftharpoons A+B \\
&2D \rightleftharpoons 2C \rightleftharpoons A+D \\
&2A \rightarrow B+C \rightleftharpoons A\\
&C+D \rightarrow \emptyset \rightleftharpoons D.
\end{align*}
Note that Theorem \ref{thm52} stands silent on this model as there is no reaction path beginning with $2B$ and ending with a unary complex or $\emptyset$.

Let $\L_1,\L_2,\L_3$ and $\L_4$ be the linkage classes of this reaction network in order from top to bottom.
We demonstrate that the assumptions of Theorem \ref{thm53} are fulfilled with $m = 2$.
\begin{enumerate}
\item The linkage classes $\L_1$ and $\L_2$ contain only binary complexes and are weakly reversible. 
\item
\begin{enumerate}
\item[(i)] For linkage class $\L_1$, we take $S = \tilde S = A$, and note the path from $2A$ to $A$ in $\L_3$.  
\item[(ii)] For linkage class $\L_2$, we take $S = C$ and $\tilde S = D$, and note the reaction $C+D \to \emptyset$ in $\L_4$.
\end{enumerate}
\item  We note $2B\in \L_1$ and $2C,2D \in \L_2$.  Also, there is a path from $2A$ to $A$ in $\L_3$.  
 \end{enumerate}
\hfill $\triangle$
\end{example}

Since weak reversibility guarantees the existence of a directed path between two complexes within any linkage class, we can modify the conditions in \thmref{thm53}.

\begin{cor}
Let $(\S,\C,\Re)$ be a weakly reversible, double-full, binary reaction network with linkage classes $\L_1,\L_2,\dots,\L_\ell$. Let $m \in \{1,\dots,\ell-1\}$ and suppose the following:
\begin{enumerate}
\item $\C(\L_i)$  contains only binary complexes for each $i \le m$ and $\C(\L_{i})$ contains at least one non-binary complex for each $i > m$. 
\item For each $i \le m$, there exist species $S, \tilde{S} \in \C(\L_i)$ such that $S+\tilde{S} \in \C(\L_j)$ for some $j > m$. 
\end{enumerate}

Then, for any choice of rate constants, every state of the Markov process with intensity function \eqref{mass} associated to the reaction network $(\S,\C,\Re)$ is positive recurrent.
\end{cor}

Now we will provide another class of double-full, binary reaction networks in which we will assume the existence of out-flows.

\begin{thm} \label{thm61}
Let $(\S,\C, \Re)$ be a double-full, binary reaction network with linkage classes  $\mathcal{L}_1,\mathcal{L}_2,\dots,\mathcal{L}_\ell$.  Suppose there is an $m \in \{1,\dots, \ell-1\}$ such that the following three conditions hold:
\begin{enumerate}
\item For each $i \le m$,  $\L_i$ is weakly reversible and $\C(\L_i)$ contains only binary complexes.
\item For each $i > m$, $C(\L_i)$ contains no binary complex.
\item For each $i \le m$, there exists an $S \in \S(\L_i)$ such that $S \to \emptyset \in \Re$. 
\end{enumerate} 
Then, for any choice of rate constants,  the Markov process with intensity functions \eqref{mass} associated to the reaction network $(\S,\C,\Re)$ satisfies the following: each state in a closed, irreducible component of the state space is positive recurrent; moreover, if  $\tau_{x_0}$ is the time for the process to enter the union of the closed irreducible components given an initial condition  $x_0$, then $\mathbb{E}[\tau_{x_0}] < \infty$.
\end{thm}

\begin{proof} 
Let $\S =\{S_1,S_2,\dots,S_d\}$ and for some $0<\delta<1$ let 
\[
T(x) = (x_1+x_2+\cdots + x_d)^{2+\delta} = (\vec 1 \cdot x)^{2+\delta},
\]
where $\vec{1}=(1,1,\cdots,1)^T \in \Z^d_{\ge0}.$
 Let $\{x_n\}$ be a tier-sequence of $(\S,\C,\Re)$.  We will show that 
\[
	\lim_{n\to \infty} \A(V+T)(x_n)=-\infty.
	\]
 An application of \thmref{thm11} then completes the proof. 
 
  We begin by finding relevant upper bounds for $\A T(x_n)$ in a similar fashion as Lemma \ref{lemma:main}.
  First, we define 
  \[
  I=\{ i \ | \ \lmt x_{n,i}=\infty\}, \quad U=\{i \ |\ S_i \to \emptyset \in \Re\}, \quad \text{and} \quad V=\{i \ | \ \emptyset \to S_i \in \Re\}.
  \]

   Since $(1+h)^{2+\delta} = 1+(2+\delta)h+o(h)\le 1+3h$ for $h$ small enough, there is a positive constant $K$ such that, for large $n$
\begin{align}
&\mathcal{A}T(x_n) =\sum_{y\rightarrow y'\in\Re}\lambda_{y\rightarrow y'}(x_n)\big( (\vec{1}\cdot x_n+\vec{1}\cdot (y'-y))^{2+\delta}-(\vec{1}\cdot x_n) ^{2+\delta}\big ) \notag\\ 
&= (\vec{1}\cdot x_n)^{2+\delta}\sum_{y\rightarrow y'\in\Re}\lambda_{y\rightarrow y'}(x_n)\left( \left( 1+\frac{\vec{1}\cdot (y'-y)}{\vec{1}\cdot x_n}\right )^{2+\delta}-1 \right) \notag \\
&\le 3(\vec{1}\cdot x_n)^{2+\delta} \sum_{y\to y' \in \Re} \lambda_{y\to y'}(x_n) \left( \frac{\vec{1} \cdot (y'-y)}{\vec{1}\cdot x_n} \right)\notag\\
& = 3(\vec{1}\cdot x_n)^{1+\delta} \sum_{y\to y' \in \Re} \lambda_{y \to y'}(x_n)(\vec{1}\cdot (y'-y)) \notag\\
&= 3(\vec{1}\cdot x_n)^{1+\delta}\left(\sum_{i \in U}\lambda_{S_i\to \emptyset}(x_n)(-1)+\sum_{i \in V}\lambda_{\emptyset\to S_i}(x_n)\right) \notag\\
&\le 3(\vec{1}\cdot x_n)^{1+\delta}\left(-\sum_{i \in U\cap I}\kappa_{S_i\to \emptyset}x_{n,i}+K \right ) \label{eq : newlya2}
\end{align}
Therefore, if there is $k \in U\cap I$ such that
\[\lmt \dfrac{(x_{n,k}\vee 1)}{(x_{n,i}\vee 1)} \ \ \ \text{exists} \ \ \text{and} \ \ \lmt \frac{(x_{n,k}\vee 1)}{(x_{n,i}\vee 1)} > 0 \]
 for all $i\in \{1,2,3,\dots, d\}$, then there is a constant $K'>0$ such that
\begin{align}\label{eq:newlya}
\mathcal{A}T(x_n) \le -K'x_{n,k}^{2+\delta}
\end{align}
for large $n$.  Hence, we have found our bound on $\A T$, and we  turn to  $\A(T+V)$.

Note that by result 2 in \lemref{lem51} there is an $i$ for which $2S_i \in \Td1$.  Without loss of generality, we assume $i = 1$ and $2S_1 \in \C(\L_1)$.
There are two cases to consider:
(i) $\C(\L_1) \subseteq \Td1$ and (ii) there is a complex $y' \in \C(\L_1)$ such that $y'\not \in \Td1$. 

\vspace{.1in}

\noindent \textbf{Case 1.} Suppose $\C(\L_1) \subseteq \Td1$. Then by  hypothesis 3, there exists a species, say $S_k$, such that $S_k \in \S(\L_1)$ and   $S_k \rightarrow \emptyset\in \Re$ . Note that $S_k+S_j \in C(\mathcal{L}_1) \subseteq \Td1$ for some $j$ (where we allow $k= j$) because $C(\mathcal{L}_1)$ only contains binary complexes. By  result 1 in Lemma \ref{lem51}, $2S_k \in \Td1$ and hence $k\in U\cap I$ by \lemref{lem : Td1 goes infinity}. Since $2S_k \in \Td1$ and the network is double-full,  for all $i=1,2,3,\dots, d$ we have
\[\lmt \dfrac{(x_{n,k}\vee 1)}{(x_{n,i}\vee 1)} \ \ \ \text{exists} \ \ \text{and} \ \ \lmt \frac{(x_{n,k}\vee 1)}{(x_{n,i}\vee 1)} > 0.
 \]
Hence, by \eqref{eq:newlya} there is a constant $K'>0$ such that for large $n$, 
\begin{align}\mathcal{A}T(x_n) \le -K'x_{n,k}^{2+\delta} 
 \label{eq:asym of AT}
\end{align}

We now turn to $\A V$.  Note that $\Td1=\Ts1$ by  result 3 in \lemref{lem51}. Since $2S_k \in \Td1=\Ts1$ there is a constant $K''>0$ such that 
\begin{align}\label{eq : asym of Ts1 and Td1 for double}
\lmt \frac{\lambda_y(x_n)}{\lambda_{2S_k}(x)} \le K'' \quad \text{and} \quad 
\lmt\frac{(x_n\vee 1)^y}{(x_n\vee 1)^{2S_k}} \le K''
\end{align}
for any complex $y\in\C$. Now applying  \lemref{lemma:main} and \eqref{eq : asym of Ts1 and Td1 for double}, we may conclude that there are positive constants $C$ and $C'$ such that
\begin{align*}
\mathcal{A}V(x_n)  \le \sum_{y\rightarrow y'\in  \Re}  \kappa_{y\to y'}\lambda_{y}(x_n) \left( \ln{\frac{(x_n\vee 1)^{y'}}{(x_n\vee 1)^y}} + C \right ) 
\le C'\lambda_{2S_k}(x_n) \ln{(x_{n,k}^2)}
\end{align*} 
for large $n$. Hence, combining our estimates for $\A T$ and $\A V$,
 \begin{align*}
\mathcal{A}(V+T)(x_n) \le C'\lambda_{2S_k}(x_n)\ln{(x_{n,k}^2)} - K'x_{n,k}^{2+\delta} \rightarrow -\infty, \quad \text{as} \ \ n\to \infty.
\end{align*} 

\vspace{.1in}

\noindent \textbf{Case 2.} We now suppose that there is a complex $y' \in \C(\L_1)$ such that $y'\not \in \Td1$. Since $\L_1$ is weakly reversible, there exists a directed path beginning with $2S_1$ and ending with $y'$. Thus there is a descending reaction $y_0 \to y_0'$ along $\{x_n\}$ within the directed path. Note that $y_0 \in \Ts1$ because $\Td1=\Ts1$ by  result 3 in \lemref{lem51} and, hence, $y_0 \in \Ts1 \cap \Des$. Since terms $I,III,IV$ in \eqref{eq:main eq for AV} are uniformly bounded in $n$ and term $II$ converges to $-\infty$, as $n \to \infty$,
\begin{align}\label{eq : AV less -lambda}
\A V(x_n) \le -\lambda_{y_0}(x_n) \quad \text{for large $n$}.
\end{align}

By \lemref{lem22} and the fact that $y_0 \in \Td1=\Ts1$, there is a constant $C\ge 0$ such that for any $ k \in I$,
\[
 \lmt \frac{x_{n,k}}{\sqrt{\lambda_{y_0}(x_n)}} = \lmt \sqrt{\frac{(x_{n,k}\vee 1)^{2S_k}}{\lambda_{y_0}(x_n)}} =
\lmt \sqrt{\frac{(x_{n,k}\vee 1)^{2S_k}}{(x_n\vee 1)^{y_0}}}
\sqrt{\frac{(x_n\vee 1)^{y_0}}{\lambda_{y_0}(x_n)}} = C.  
 \]
Therefore, there is a constant $C' > 0$ such that
\begin{align}\label{eq : AT is bounded by AV}
(\vec{1}\cdot x_n)^{1+\delta} = (x_{n,1}+x_{n,2}+\dots+x_{n,d})^{1+\delta} \le C'' \lambda_{y_0}(x_n)^{(1+\delta)/2}.
\end{align}
Note that $\lambda_{y_0}(x_0) \to \infty$, as $n\to \infty$ by Corollary \ref{cor23}. Then by \eqref{eq : newlya2},\eqref{eq : AV less -lambda} and \eqref{eq : AT is bounded by AV}, there are constants $C''' >0$ such that
\begin{align*}
\A(V+T)(x_n) \le -\lambda_{y_0}(x_n)+C'''\lambda_{y_0}(x_n)^{(1+\delta)/2} \to \infty, \quad \text{as } \ n \to \infty,
\end{align*}
because $\delta < 1$.
\end{proof}
We demonstrate \thmref{thm61} with an example.
\begin{example}
Consider the following double-full, binary reaction network with 16 reactions.
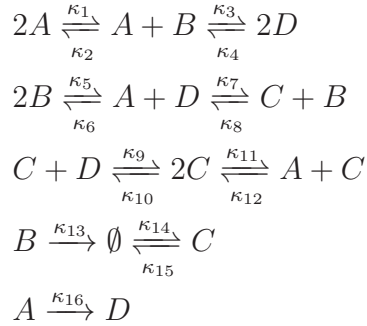
\begin{align*}
&2A \xrightleftharpoons[\kappa_2]{\kappa_1} A+B \xrightleftharpoons[\kappa_4]{\kappa_3} 2D\\
&2B \xrightleftharpoons[\kappa_6]{\kappa_5} A+D \xrightleftharpoons[\kappa_8]{\kappa_7} C+B\\
&C+D \xrightleftharpoons[\kappa_{10}]{\kappa_9} 2C \xrightleftharpoons[\kappa_{12}]{\kappa_{11}} A+C\\
&B \xrightarrow{\kappa_{13}} \emptyset \xrightleftharpoons[\kappa_{15}]{\kappa_{14}} C \\
&A \xrightarrow{\kappa_{16}} D
\end{align*}
Let $\L_1,\dots, \L_5$ be the linkage classes of this reaction network in order from top to bottom.  We verify the conditions of Theorem \ref{thm61} with $m = 3$.
\begin{enumerate}
\item   $\mathcal{L}_1$, $\mathcal{L}_2$, and $\mathcal{L}_3$ are weakly reversible and contain only binary complexes.
\item   $\mathcal{L}_4$ and $\L_5$ do not contain binary complexes.
\item  Note that species  $B \in S(\mathcal{L}_1)$, $B \in S(\mathcal{L}_2)$ and $C \in S(\mathcal{L}_3)$ satisfy the third condition of the theorem. 
\end{enumerate}
Moreover, $\mathbb{S} = \Z^4_{\ge 0}$ is irreducible.  Therefore the associate continuous-time Markov chain for this reaction network is positive recurrent for any choice of rate constants $\kappa_1,\kappa_2,\dots,\kappa_{16}$.
  \hfill $\triangle$
\end{example}

\section{Discussion}
\label{sec:discussion}
We introduced two main  network conditions, each of which guarantee that each state in a closed, irreducible component of the state space is positive recurrent, and that, regardless of initial condition, all trajectories enter a closed, irreducible component in finite time.    The analysis was based on the idea of \textit{tiers} introduced in \cite{AndGAC_ONE2011,AndBounded_ONE2011}.  There are a number of avenues for future work.

First, we believe that \textit{all} systems whose reaction networks are weakly reversible are positive recurrent, and the present work grew out of an attempt to prove this.  Specifically, we believe the following to be true.

\vspace{.1in}

 C{\footnotesize ONJECTURE} (Positive Recurrence Conjecture).  \textit{Let $(\S,\C,\Re)$ be a weakly reversible reaction network.  Then, for any choice of rate constants, every state of the Markov process with intensity function \eqref{mass} associated to the reaction network $(\S,\C,\Re)$ is positive recurrent.}

\vspace{.1in}

Attempting to prove this conjecture, which is closely related to the Global Attractor Conjecture \cite{Craciun2015,CraciunShiu09} for deterministic models,  remains an active pursuit.

Second, a natural follow-up question is: how fast do the processes considered here converge to their stationary distributions.  Results related to this question will be presented in a forthcoming paper.

 \bibliographystyle{plain}
\bibliography{library}

\end{document}